\documentclass[12pt]{amsart}

\usepackage[utf8]{inputenc}
\usepackage[a4paper,nomarginpar]{geometry}
\usepackage[english]{babel}
\usepackage{enumitem}
\usepackage{amsmath,amsfonts,amssymb,amsthm,amscd}
\usepackage{tikz}

\allowdisplaybreaks

\makeatletter
\def\subsection{\@startsection{subsection}{2}%
    \z@{.5\linespacing\@plus.7\linespacing}{.3\linespacing}%
    {\normalfont\bfseries}}
\makeatother

\theoremstyle{theorem}
\newtheorem{theorem}{Theorem}

\newtheorem{proposition}[theorem]{Proposition}
\newtheorem{corollary}[theorem]{Corollary}

\theoremstyle{definition}
\newtheorem{definition}[theorem]{Definition}
\newtheorem{example}[theorem]{Example}
\newtheorem{remark}[theorem]{Remark}

\theoremstyle{remark} \theoremstyle{question} \theoremstyle{example}

\newcommand{\N}{\mathbb{N}}   
\newcommand{\Z}{\mathbb{Z}}   
\newcommand{\R}{\mathbb{R}}   
\newcommand{\C}{\mathbb{C}}   
\newcommand{\K}{\mathbb{K}}   

\newcommand{\cP}{\mathcal{P}}

\newcommand{\vertiii}[1]{{\left\vert\kern-0.25ex\left\vert\kern-0.25ex\left\vert #1 
    \right\vert\kern-0.25ex\right\vert\kern-0.25ex\right\vert}}

\newcommand{\eps}{\varepsilon}   
\newcommand{\wt} {\widetilde}    

\newcommand{\card}{\operatorname{card}}

\newcommand{\udens}{\operatorname{\overline{dens}}}

\begin{document}

\title[On notions of expansivity for operators on locally convex spaces]
{On notions of expansivity for operators\\ on locally convex spaces}

\date{}
\dedicatory{}
\maketitle

\begin{center}
{\sc Nilson C. Bernardes Jr.}$^\text{a,b}$, \ {\sc F\'elix Mart\'inez-Gim\'enez}$\,^\text{a}$,\\ {\sc and Francisco Rodenas}$\,^\text{a}$
\end{center}

\bigskip\bigskip
{\small\noindent
$^\text{a}$ {\it Institut Universitari de Matem\`atica Pura i Aplicada, Universitat Polit\`ecnica de Val\`encia,
Cam\'i de Vera S/N, Edifici 8E, Acces F, 4a Planta, Val\`encia, 46022, Spain.}

\smallskip\noindent
$^\text{b}$ {\it Departamento de Matem\'atica Aplicada, Instituto de Matem\'atica, Universidade Federal do Rio de Janeiro,
Caixa Postal 68530, Rio de Janeiro, 21945-970, Brazil.}

\smallskip\noindent
{\it e-mails:} ncbernardesjr@gmail.com, fmartinez@mat.upv.es, frodenas@mat.upv.es.}

\bigskip

\begin{abstract}
We extend the concept of average expansivity for operators on Banach spaces to operators on arbitrary locally convex spaces.
We obtain complete characterizations of the average expansive weighted shifts on Fr\'echet sequence spaces.
Moreover, we give a partial answer to a problem proposed by Bernardes et al.\ in J. Funct.\ Anal.\ {\bf 288} (2025), Paper No.\ 110696,
by obtaining complete characterizations of the uniformly expansive weighted shifts on K\"othe sequence spaces.
Some general properties of the various concepts of expansivity in linear dynamics and several concrete examples are also presented.
\end{abstract}


\section{Introduction}

Given a metric space $M$ with metric $d$, recall that a homeomorphism $h : M \to M$ is said to be {\em expansive} 
if there is a constant $c > 0$ (called an {\em expansivity constant} for $h$) such that for any pair $x,y$ of distinct points in $M$, 
there exists $n \in \Z$ satisfying
\[
d(h^n(x),h^n(y)) > c.
\]
This concept was introduced by Utz~\cite{WUtz50} in 1950. It has played an important role in the modern theory of dynamical systems,
having many applications and connections to other fundamental concepts in dynamics.
We refer the reader to the book~\cite{NAokKHir94} for a nice exposition about this concept.

In the setting of linear dynamics, Eisenberg~\cite{MEis66} proved that an invertible operator on $\C^n$ is expansive if and only if it has no eigenvalue on the unit circle. 
A little later, in the early 1970s, Eisenberg and Hedlund~\cite{MEisJHed70} and Hedlund~\cite{JHed71} studied the concept of expansivity for operators on Banach spaces. 
They introduced the concept of uniform expansivity and proved that an invertible operator on a Banach space is uniformly expansive if and only if its approximate point spectrum does not intersect the unit circle. 
Later, in 2000, Mazur~\cite{MMaz00} proved that an invertible normal operator $T$ on a Hilbert space is expansive if and only if $T^*T$ has no eigenvalue on the unit circle. 
This subject was revisited by Bernardes et al.~\cite{BerCirDarMesPuj18} in 2018.
Among the many results contained in~\cite{BerCirDarMesPuj18}, we can find complete characterizations of the various notions of expansivity for weighted shifts on classical Banach sequence spaces, the fact that uniformly expansive operators on Banach spaces are never Li-Yorke chaotic, and the fact that an invertible operator on a Banach space is uniformly positively expansive if and only if its approximate point spectrum does not intersect the closed unit disc. 
Subsequently, it was proved that an invertible operator on a Banach space is hyperbolic if and only if it is expansive and has the shadowing property~\cite{NBerAMes21,CirGolPuj21}.
Alves et al.~\cite{AlvBerMes21} introduced the concept of average expansivity for operators on Banach spaces and characterized the average expansive weighted shifts on classical Banach sequence spaces.
Some additional results on expansive operators can be found in the recent papers~\cite{AniDarMai22,KLeeCMor22,MMai22}.
Finally, Bernardes et al.~\cite{BCDFP} have recently extended the concepts of expansivity and uniform expansivity to operators on locally convex spaces and initiated investigations on these concepts in this much more general setting.

Our goal in this note is to complement these previous works on expansivity in the setting of linear dynamics, 
especially the papers~\cite{AlvBerMes21,BCDFP,BerCirDarMesPuj18}.

In Section~\ref{S-AE} we extend the concept of average expansivity for operators on Banach spaces \cite{AlvBerMes21} 
to operators on arbitrary locally convex spaces (Definition~\ref{AE-D1}).
Our main result in this section is a complete characterization of the average expansive weighted shifts 
on Fr\'echet sequence spaces (Theorem~\ref{AE-T2}).
Applications to an important special class of Fr\'echet sequence spaces, 
known as K\"othe sequence spaces, are also provided (Corollary~\ref{AE-C1} and Remark~\ref{AE-R1}). 
Due to the importance of weighted shifts in Operator Theory, 
the dynamics of these operators has been extensively studied by many authors.
The classical references are \cite{KGro00} and \cite{HSal95}.
Recent contributions to this theme include \cite{NBerAMes21} and \cite{NBerAPer2024}, where additional references can be found.
In the special case of K\"othe sequence spaces, we refer the reader to \cite{BCDFP,NBerAPer2024,NBerFVas25,FMarAPer02,XWuPZhu13},
for instance.
Although uniformly expansive operators on locally convex spaces are neither Li-Yorke chaotic nor topologically transitive
\cite{BCDFP}, we show that there exist average expansive weighted shifts on classical Banach sequence spaces
that are simultaneously distributionally chaotic and topologically transitive (Theorem~\ref{AE-T3}).

In Section~\ref{S-UE-K} we address a problem proposed by Bernardes et al.\ \cite{BCDFP}.
Complete characterizations of the notions of expansivity and uniform expansivity for weighted shifts
on classical Banach sequence spaces were obtained by Bernardes et al.\ \cite{BerCirDarMesPuj18}.
For the notion of expansivity, the characterization was generalized to weighted shifts on Fr\'echet sequence spaces in \cite{BCDFP},
but the following problem was left open:
{\it Characterize the concept of uniform expansivity for weighted shifts on Fr\'echet sequence spaces} \cite[Problem~E]{BCDFP}.
We present here a partial answer to this problem by obtaining complete characterizations of the uniformly expansive
weighted shifts on K\"othe sequence spaces (Theorem~\ref{UEWS-T1}, Remarks~\ref{UEWS-R2} and~\ref{UEWS-R3}).
Moreover, contrary to what happens in the case of Banach spaces, we give an example of a weighted shift on the K\"othe sequence space
$s(\Z)$ of all rapidly decreasing sequences whose trajectories do not grow exponentially fast (Example~\ref{UEWS-E2}).

In Section~\ref{S-Gen} we present properties of the various notions of expansivity related to inverses, rotations, powers, products,
direct sums and restrictions (Propositions~\ref{Gen-P1}, \ref{Gen-P2} and~\ref{Gen-P3}).
This section is elementary in nature and is included here for the purpose of organizing these basic facts for future reference 
and for the sake of completeness.


\section{Average expansive operators on locally convex spaces}\label{S-AE}

Throughout $\K$ denotes either the field $\R$ of real numbers or the field $\C$ of complex numbers, $\Z$ denotes the ring of integers, 
$\N$ denotes the set of all positive integers, $-\N$ denotes the set of all negative integers, and $\N_0 = \N \cup \{0\}$.

Given a seminorm $\|\cdot\|$ on a vector space $X$, the {\em unit sphere} of $\|\cdot\|$ is defined by
\[
S_{\|\cdot\|} = \{x \in X : \|x\| = 1\}.
\]
If $X$ is a normed space with norm $\|\cdot\|$, then we also write $S_X$ instead of $S_{\|\cdot\|}$.

Let $X$ be a locally convex space over $\K$ (all spaces are assumed to be Hausdorff) and choose a family 
$(\|\cdot\|_\alpha)_{\alpha \in I}$ of seminorms that induces its topology.
We denote by $L(X)$ the set of all continuous linear operators (simply, {\em operators}) on $X$
and by $GL(X)$ the set of those operators that have a continuous inverse.

Recall that an operator $T \in GL(X)$ is said to be {\em topologically expansive} if the following condition holds:
\begin{itemize}
\item [(E)] For each $x \in X \backslash \{0\}$, there exists $\alpha \in I$ such that $\sup_{n \in \Z} \|T^n x\|_\alpha = \infty$.
\end{itemize}
Moreover, $T$ is said to be {\em uniformly topologically expansive} if:
\begin{itemize}
\item [(UE)] For each $\alpha \in I$, there exists $\beta \in I$ such that we can write 
$S_{\|\cdot\|_\alpha} = S^+_\alpha \cup S^-_\alpha$, where
\[
\|T^n x\|_\beta \to \infty \text{ uniformly on } S^+_\alpha \text{ as } n \to \infty
\]
and
\[
\|T^{-n} x\|_\beta \to \infty \text{ uniformly on } S^-_\alpha \text{ as } n \to \infty.
\]
\end{itemize}

These concepts were introduced in \cite{BCDFP} as generalizations of the classical notions of expansivity and uniform expansivity for operators on Banach spaces (see \cite{MEisJHed70}, for instance).
The word ``topologically'' was used to distinguish the above concept of expansivity from the usual notion of expansivity in the metric space setting, since they may differ in the case of a metrizable locally convex space (or even a Fr\'echet space) endowed with a compatible invariant metric \cite[Example~31]{BCDFP}, although they always coincide in normed spaces.
However, in the present note we will omit the word ``topologically'' and will write simply ``expansive'' and ``uniformly expansive''
for the above concepts.

A concept of average expansivity for operators on Banach spaces was introduced and investigated in \cite{AlvBerMes21}.
Below we generalize it by means of the following definition.

\begin{definition}\label{AE-D1}
Let $X$ be a locally convex space over $\K$ whose topology is induced by a family $(\|\cdot\|_\alpha)_{\alpha \in I}$ of seminorms.
We say that $T \in GL(X)$ is {\em average expansive} if the following condition holds:
\begin{itemize}
\item [(AE)] For each $x \in X \backslash \{0\}$, there exists $\alpha \in I$ such that
\[
\sup_{n \in \N} \Big(\frac{1}{2n+1} \sum_{j=-n}^n \|T^j x\|_\alpha \Big) = \infty.
\]
\end{itemize}
\end{definition}

We observe that this concept does not depend on the choice of the family of seminorms inducing the topology of $X$. 
The same is true for the concepts of expansivity and uniform expansivity, as observed in~\cite{BCDFP}.
In particular, we can take the family of all continuous seminorms on $X$ if we wish so.
This also shows that these notions are completely determined by the topology of the locally convex space $X$.

\begin{proposition}\label{AE-P1}
For every $T \in GL(X)$,
\[
T \text{ is uniformly expansive } \Longrightarrow \ T \text{ is average expansive } \Longrightarrow \ T \text{ is expansive}.
\]
\end{proposition}

\begin{proof}
The second implication follows from the inequality
\[
\sup_{n \in \N} \Big(\frac{1}{2n+1} \sum_{j=-n}^n \|T^j x\|_\alpha \Big) \leq \sup_{n \in \Z} \|T^nx\|_\alpha.
\]

In order to prove the first implication, suppose that the operator $T$ is uniformly expansive. 
Given $x \in X \backslash \{0\}$, there exists $\alpha \in I$ such that $\|x\|_\alpha \neq 0$.
Let $\beta \in I$ be associated to $\alpha$ according to the definition of uniform expansivity. Then
\[
\lim_{n \to \infty} \|T^n x\|_\beta = \infty \ \ \text{ or } \ \ \lim_{n \to \infty} \|T^{-n}x\|_\beta = \infty.
\]
Let us assume the first possibility (the argument for the second one is analogous). 
Then, for each $k \in \N$, there exists $n_k \in \N$ such that $\|T^n x\|_\beta \geq k$ whenever $n \geq n_k$. Hence,
\[
\sup_{n \in \N} \Big( \frac{1}{2n+1} \sum_{j=-n}^n \|T^j x\|_\beta \Big)
\geq \sup_{n > n_k} \Big(\frac{1}{2n+1} \sum_{j=n_k}^n \|T^j x\|_\beta \Big)
\geq \sup_{n > n_k} \frac{k(n-n_k)}{2n+1}
= \frac{k}{2}\cdot
\]
Since $k \in \N$ is arbritary, we conclude that $T$ is average expansive.
\end{proof}

Our next goal is to characterize the average expansive weighted shifts on Fr\'echet sequence spaces. For this purpose,
recall that a {\em Fr\'echet sequence space over $\Z$} is a Fr\'echet space $X$ which is a vector subspace of the product space 
$\K^\Z$ such that the inclusion map $X \to \K^\Z$ is continuous, that is, convergence in $X$ implies coordinatewise convergence.
Given a sequence $w = (w_j)_{j \in \Z}$ of nonzero scalars, it follows from the closed graph theorem that
the {\em bilateral weighted backward shift}
\[
B_w((x_j)_{j \in \Z}) = (w_{j+1}x_{j+1})_{j \in \Z}
\]
is a continuous linear operator on $X$ as soon as it maps $X$ into itself.
The same is true for the {\em bilateral weighted forward shift}
\[
F_w((x_j)_{j \in \Z}) = (w_{j-1}x_{j-1})_{j \in \Z}.
\]
The {\em canonical vectors} of $\K^\Z$ are the vectors $e_j$, $j \in \Z$, where the $j^\text{th}$ coordinate of $e_j$ is $1$
and the other coordinates of $e_j$ are $0$.
The sequence $(e_j)_{j \in \Z}$ is a {\em basis} of $X$ if each $e_j$ belongs to $X$ and
\[
x = \sum_{j=-\infty}^\infty x_j e_j \ \ \text{ for all } x = (x_j)_{j \in \Z} \in X.
\]

If we replace $\Z$ by $\N$ in the previous paragraph, then we obtain the concept of a {\em Fr\'echet sequence space} $X$,
the {\em unilateral weighted backward shift}
\[
B_w(x_1,x_2,x_3,\ldots) = (w_2x_2,w_3x_3,w_4x_4,\ldots),
\]
and the {\em unilateral weighted forward shift}
\[
F_w(x_1,x_2,x_3,\ldots) = (0,w_1x_1,w_2x_2,w_3x_3,\ldots).
\]
As before, $B_w$ (resp.\ $F_w$) is a continuous linear operator on $X$ as soon as it maps $X$ into itself.
By abuse of language, we also denote the {\em canonical vectors} of $\K^\N$ by $e_j$, $j \in \N$.

\begin{theorem}\label{AE-T1}
Suppose that $X$ is a Fréchet sequence space over $\Z$ in which the sequence $(e_j)_{j \in \Z}$ of canonical vectors is a basis,
$(\|\cdot\|_k)_{k \in \N}$ is an increasing sequence of seminorms that induces the topology of $X$,
and the bilateral backward shift 
\[
B : (x_j)_{j \in \Z} \in X \mapsto (x_{j+1})_{j \in \Z} \in X
\]
is an invertible operator on $X$. Then $B$ is average expansive if and only if there exists $k \in \N$ such that
\begin{equation}\label{AE-F1}
\sup_{n \in \N} \frac{\|e_{-1}\|_k + \cdots + \|e_{-n}\|_k}{n} = \infty \ \ \text{ or } \ \
\sup_{n \in \N} \frac{\|e_1\|_k + \cdots + \|e_n\|_k}{n} = \infty.
\end{equation}
\end{theorem}

\begin{proof}
Suppose that, for each $k \in \N$, there is a constant $C_k \in (0,\infty)$ such that
\[
\frac{\|e_{-1}\|_k + \cdots + \|e_{-n}\|_k}{n} \leq C_k \ \text{ and } \
\frac{\|e_1\|_k + \cdots + \|e_n\|_k}{n} \leq C_k \ \text{ for all } n \in \N.
\]
Then,
\[
\frac{1}{2n+1} \sum_{j=-n}^n \|B^j e_0\|_k \leq \frac{2 C_k n + \|e_0\|_k}{2n+1} \to C_k \ \text{ as } n \to \infty,
\]
showing that condition (AE) fails for the vector $e_0$.

Conversely, suppose that there is a vector $y = (y_j)_{j \in \Z} \in X \backslash \{0\}$ such that
\[
C_k = \sup_{n \in \N} \Big(\frac{1}{2n+1} \sum_{j=-n}^n \|B^j y\|_k \Big) < \infty \ \ \text{ for all } k \in \N.
\]
Since $(e_n)_{n \in \Z}$ is a basis of $X$, the Banach-Steinhaus theorem ensures that the family of operators
\[
x \in X \mapsto x_j e_j \in X \ \ \ \ (j \in \Z)
\]
is equicontinuous. Hence, for each $k \in \N$, there exist $D_k \in (0,\infty)$ and $m_k \in \N$ such that
\[
\|x_je_j\|_k \leq D_k \|x\|_{m_k} \ \ \text{ whenever } x \in X \text{ and } j \in \Z.
\]
Since $y \neq 0$, there exists $i \in \Z$ such that $y_i \neq 0$.
If $n > |i|$ and $t_n = \max\{n-i,n+i\}$, then
\begin{align*}
\frac{1}{n} \sum_{j=-n}^{n} \|e_j\|_k 
&= \frac{1}{n|y_i|} \sum_{j=-n}^{n} \|y_ie_j\|_k 
  \leq \frac{D_k}{n|y_i|} \sum_{j=-n}^{n} \|B^{j+i}y\|_{m_k}
  \leq \frac{D_k}{n|y_i|} \sum_{j=-t_n}^{t_n} \|B^jy\|_{m_k}\\
&\leq \frac{C_{m_k}D_k}{|y_i|} \cdot \frac{2t_n + 1}{n}
  \ \to \ \frac{2C_{m_k}D_k}{|y_i|} \ \text{ as } n \to \infty.
\end{align*}
Thus,
\[
\sup_{n \in \N} \Big(\frac{1}{n} \sum_{j=-n}^{n} \|e_j\|_k\Big) < \infty \ \ \text{ for all } k \in \N,
\]
which implies that (\ref{AE-F1}) is false.
\end{proof}

The previous theorem can be generalized to bilateral weighted backward shifts as follows.

\begin{theorem}\label{AE-T2}
Suppose that $X$ is a Fréchet sequence space over $\Z$ in which the sequence $(e_j)_{j \in \Z}$ of canonical vectors is a basis,
$(\|\cdot\|_k)_{k \in \N}$ is an increasing sequence of seminorms that induces the topology of $X$,
$w = (w_j)_{j \in \Z}$ is a sequence of nonzero scalars, and the bilateral weighted backward shift 
\[
B_w : (x_j)_{j \in \Z} \in X \mapsto (w_{j+1} x_{j+1})_{j \in \Z} \in X
\]
is an invertible operator on $X$. Then $B_w$ is average expansive if and only if there exists $k \in \N$ such that
\begin{equation}\label{AE-F2}
\sup_{n \in \N} \Big(\frac{1}{n} \sum_{j=1}^n |w_{-j+1} \cdots w_0| \|e_{-j}\|_k\Big) = \infty \ \ \text{ or } \ \
\sup_{n \in \N} \Big(\frac{1}{n} \sum_{j=1}^n \frac{\|e_j\|_k}{|w_1 \cdots w_j|}\Big) = \infty.
\end{equation}
\end{theorem}

\smallskip
This theorem follows from the previous one by means of a suitable conjugacy.
The method can be found in \cite[Section~4.1]{KGroAPer11}, but we will recall it here briefly for the sake of completeness.
Consider the weights
\[
v_0 = 1, \ \ \ \ v_{-j} = w_{-j+1} \cdots w_0 \ \ \text{ and } \ \ v_j = \frac{1}{w_1 \cdots w_j} \ \text{ for } j \geq 1,
\]
the vector space
\[
X_v = \{(x_j)_{j \in \Z} \in \K^\Z : (v_j x_j)_{j \in \Z} \in X\},
\]
and the vector space isomorphism
\[
\phi_v : (x_j)_{j \in \Z} \in X_v \mapsto (v_j x_j)_{j \in \Z} \in X.
\]
Use $\phi_v$ to transfer the seminorms of $X$ to $X_v$:
\[
\|(x_j)_{j \in \Z}\|'_k = \|\phi_v((x_j)_{j \in \Z})\|_k \ \ \text{ for } (x_j)_{j \in \Z} \in X_v.
\]
Endowed with the topology induced by the sequence $(\|\cdot\|'_k)_{k \in \N}$ of seminorms,
$X_v$ is a Fr\'echet sequence space in which $(e_j)_{j \in \Z}$ is a basis.
Since $B_w \circ \phi_v = \phi_v \circ B$, we have that $B_w$ is average expansive if and only if so is $B$.
Thus, Theorem~\ref{AE-T2} follows from Theorem~\ref{AE-T1} applied to $X_v$.

An important class of Fr\'echet sequence spaces is that of K\"othe sequence spaces. Let $J = \N$ or~$\Z$. 
Recall that a {\em K\"othe matrix} ({\em on $J$}) is a family of scalars of the form 
$A = (a_{j,k})_{j \in J,k \in \N}$ such that the following properties hold for each $j \in J$:
\begin{itemize}
\item $0 \leq a_{j,1} \leq a_{j,2} \leq \cdots \leq a_{j,k} \leq a_{j,k+1} \leq \cdots$.
\item $a_{j,k} > 0$ for some $k \in \N$. 
\end{itemize}
Given a real number $p \in \{0\} \cup [1,\infty)$ and such a K\"othe matrix $A$, 
the associated {\em K\"othe sequence space} $\lambda_p(A,J)$ is the Fr\'echet sequence space defined as follows:
\begin{itemize}
\item If $p \in [1,\infty)$, then 
    \[
    \lambda_p(A,J) = \Big\{(x_j)_{j \in J} \in \K^J : \sum_{j \in J} |a_{j,k} x_j|^p < \infty \text{ for all } k \in \N\Big\}
    \] 
    endowed with the seminorms
    \[
    \|x\|_k = \Big(\sum_{j \in J} |a_{j,k} x_j|^p\Big)^\frac{1}{p} \ \text{ for } x = (x_j)_{j \in J} \in \lambda_p(A,J) \text{ and } k \in \N.
    \]
\item If $p = 0$, then 
    \[
    \lambda_p(A,J) = \Big\{(x_j)_{j \in J} \in \K^J : \lim_{j \in J, |j| \to \infty} a_{j,k} x_j = 0 \text{ for all } k \in \N\Big\}
    \]
    endowed with the seminorms
    \[
    \|x\|_k = \sup_{j \in J} |a_{j,k} x_j| \ \text{ for } x = (x_j)_{j \in J} \in \lambda_p(A,J) \text{ and } k \in \N.
    \]
\end{itemize}
Note that the sequence $(e_j)_{j \in J}$ of canonical vectors in $\K^J$ is clearly a basis of $\lambda_p(A,J)$ 
for every $p \in \{0\} \cup [1,\infty)$.
We refer the reader to the books \cite{HJar81,GKot69,RMeiDVog97} for a detailed study of K\"othe sequence spaces.

Let us adopt the convention that $\frac{0}{0} = 1$. It is well known that a bilateral weighted backward shift $B_w$ is an operator 
on $\lambda_p(A,\Z)$ (i.e., it maps $\lambda_p(A,\Z)$ into itself) if and~only if, for each $k \in \N$, 
there exists $\ell \in \N$ such that $a_{j,k} = 0$ whenever $a_{j+1,\ell} = 0$ ($j \in \Z$),~and  
\begin{equation}\label{Bw-Defined}
\sup_{j \in \Z} \frac{a_{j,k} |w_{j+1}|}{a_{j+1,\ell}} < \infty.
\end{equation}
In this case, $B_w$ is invertible if and only if, for each $k \in \N$, there exists $\ell \in \N$ such that 
$a_{j+1,k} = 0$ whenever $a_{j,\ell} = 0$ ($j \in \Z$), and  
\begin{equation}\label{Bw-Invertible}
\sup_{j \in \Z} \frac{a_{j+1,k}}{a_{j,\ell} |w_{j+1}|} < \infty.
\end{equation}

Theorem~\ref{AE-T2} applied to K\"othe sequence spaces gives the following result.

\begin{corollary}\label{AE-C1}
Consider a K\"othe sequence space $X = \lambda_p(A,\Z)$, where $A = (a_{j,k})$ is a K\"othe matrix on $\Z$ 
and $p \in \{0\} \cup [1,\infty)$. 
Let $w = (w_j)_{j \in \Z}$ be a sequence of nonzero scalars such that the bilateral weighted backward shift
\[
B_w : (x_j)_{j \in \Z} \in X \mapsto (w_{j+1}x_{j+1})_{j \in \Z} \in X
\]
is an invertible operator on $X$ (i.e., conditions (\ref{Bw-Defined}) and (\ref{Bw-Invertible}) hold). 
Then $B_w$ is average expansive if and only if there exists $k \in \N$ such that
\begin{equation}\label{AE-F3}
\sup_{n \in \N} \Big(\frac{1}{n} \sum_{j=1}^n a_{-j,k}|w_{-j+1} \cdots w_0|\Big) = \infty \ \ \text{ or } \ \
\sup_{n \in \N} \Big(\frac{1}{n} \sum_{j=1}^n \frac{a_{j,k}}{|w_1 \cdots w_j|}\Big) = \infty.
\end{equation}
\end{corollary}

\begin{example}
If $a_{j,k} = 1$ for all $j \in \Z$ and $k \in \N$, then
\[
\lambda_0(A,\Z) = c_0(\Z) \ \ \text{ and } \ \ \lambda_p(A,\Z) = \ell^p(\Z) \text{ for } p \in [1,\infty).
\]
In this case, Corollary~\ref{AE-C1} recovers \cite[Proposition~15]{AlvBerMes21}.
We can also recover \cite[Proposition~15]{AlvBerMes21} directly from Theorem~\ref{AE-T2}.
\end{example}

\begin{remark}\label{AE-R1}
Suppose that a bilateral weighted forward shift $F_w$ is an invertible operator on a Fr\'echet sequence space $X$ 
as in Theorem~\ref{AE-T2}. Since $F_w$ is average expansive if and only if so is $F_w^{-1}$, and since
\[
F_w^{-1} = B_{w'}, \ \text{ where } w' = \Big(\frac{1}{w_{j-1}}\Big)_{j \in \Z},
\]
it follows from Theorem~\ref{AE-T2} that $F_w$ is average expansive if and only if there exists $k \in \N$ such that
\begin{equation}\label{AE-F4}
\sup_{n \in \N} \Big(\frac{1}{n} \sum_{j=1}^n |w_0 \cdots w_{j-1}| \|e_j\|_k\Big) = \infty \ \ \text{ or } \ \
\sup_{n \in \N} \Big(\frac{1}{n} \sum_{j=1}^n \frac{\|e_{-j}\|_k}{|w_{-j} \cdots w_{-1}|}\Big) = \infty.
\end{equation}
In the particular case of K\"othe sequence spaces, it is well-known that $F_w$ is an operator on $\lambda_p(A,\Z)$ if and only if, 
for each $k \in \N$, there exists $\ell \in \N$ such that $a_{j+1,k} = 0$ whenever $a_{j,\ell} = 0$ ($j \in \Z$), and  
\begin{equation}\label{Fw-Defined}
\sup_{j \in \Z} \frac{a_{j+1,k} |w_j|}{a_{j,\ell}} < \infty.
\end{equation}
In this case, $F_w$ is invertible if and only if, for each $k \in \N$, there exists $\ell \in \N$ such that 
$a_{j,k} = 0$ whenever $a_{j+1,\ell} = 0$ ($j \in \Z$), and  
\begin{equation}\label{Fw-Invertible}
\sup_{j \in \Z} \frac{a_{j,k}}{a_{j+1,\ell} |w_j|} < \infty.
\end{equation}
Assuming that $F_w$ is an invertible operator on $\lambda_p(A,\Z)$ (i.e., conditions (\ref{Fw-Defined}) and (\ref{Fw-Invertible}) hold), 
we have that $F_w$ is average expansive if and only if there exists $k \in \N$ such that
\begin{equation}\label{AE-F5}
\sup_{n \in \N} \Big(\frac{1}{n} \sum_{j=1}^n a_{j,k} |w_0 \cdots w_{j-1}|\Big) = \infty \ \ \text{ or } \ \
\sup_{n \in \N} \Big(\frac{1}{n} \sum_{j=1}^n \frac{a_{-j,k}}{|w_{-j} \cdots w_{-1}|}\Big) = \infty.
\end{equation}
\end{remark}

\begin{remark}\label{AE-R2}
For operators that are not necessarily invertible, we can consider the following variation of the concept of average expansivity:
$T \in L(X)$ is said to be {\em average positively expansive} if the following condition holds:
\begin{itemize}
\item [(APE)] For each $x \in X \backslash \{0\}$, there exists $\alpha \in I$ such that
\[
\sup_{n \in \N} \Big(\frac{1}{n} \sum_{j=0}^{n-1} \|T^j x\|_\alpha \Big) = \infty.
\]
\end{itemize}
In the invertible case, it is easy to see that
\[
T \text{ or } T^{-1} \text{ is average positively expansive } \ \Longrightarrow \ \ T \text{ is average expansive}.
\]
The converse is easily seen to be false in general (e.g., $T : (x,y) \in \R^2 \mapsto (2x,y/2) \in \R^2$).
However, the converse is true in the case of invertible bilateral weighted shifts.
This follows from the fact that {\it the validity of the first (resp.\ second) equality in (\ref{AE-F2}), for some $k \in \N$,
is equivalent to $B_w$ (resp.\ $B_w^{-1}$) being average positively expansive}.
The same fact holds for $F_w$ provided we replace (\ref{AE-F2}) by (\ref{AE-F4}).
For unilateral weighted forward shifts, the characterization of average positive expansivity is analogous:
{\it A unilateral weighted forward shift $F_w$ on a Fr\'echet sequence space $X$ in which $(e_j)_{j \in \N}$ is a basis
is average positively expansive if and only if there exists $k \in \N$ such that
\[
\sup_{n \in \N} \Big(\frac{1}{n} \sum_{j=1}^n |w_1 \cdots w_{j-1}| \|e_j\|_k\Big) = \infty,
\]
where $(\|\cdot\|_k)_{k \in \N}$ is an increasing sequence of seminorms that induces the topology of $X$
and we are adopting the convention that $|w_1 \cdots w_{j-1}| = 1$ if $j = 1$.}
\end{remark}

Recall \cite{BCDFP} that an operator $T$ on a topological vector space $X$ is said to be {\em Li-Yorke chaotic}
if there is an uncountable set $S \subset X$ such that each pair $(x,y)$ of distinct points in $S$ is a {\em Li-Yorke pair} for $T$,
in the sense that the following conditions hold:
\begin{itemize}
\item [(LY1)] For every neighborhood $V$ of $0$ in $X$, there exists $n \in \N$ such that $T^nx - T^ny \in V$.
\item [(LY2)] There exists a neighborhood $U$ of $0$ in $X$ such that $T^nx - T^ny \not\in U$ for infinitely many values of $n$.
\end{itemize}
Recall also that $T$ is said to be {\em topologically transitive} (resp.\ {\em topologically mixing}) if 
for any pair $A,B$ of nonempty open sets in $X$, there exists $n \in \N_0$ (resp.\ $n_0 \in \N_0$) such that
$T^n(A) \cap B \neq \emptyset$ (resp.\ for all $n \geq n_0$).

It was proved in \cite[Theorem~38]{BCDFP} that uniformly expansive operators on locally convex spaces
are neither Li-Yorke chaotic nor topologically transitive.
The next result shows that this fact is no longer true if we replace uniform expansivity by average expansivity.
In fact, average expansive operators on Banach spaces can even be distributionally chaotic.

Recall that the {\em upper density} of a set $A \subset \N$ is defined by
\[
\udens(A) = \limsup_{n \to \infty} \frac{\card(A \cap [1,n])}{n}\cdot
\]
An operator $T$ on a Banach space $X$ is said to be {\em distributionally chaotic} if there exist an uncountable set $\Gamma \subset X$
and $\eps > 0$ such that each pair $(x,y)$ of distinct points in $\Gamma$ is a {\em distributionally chaotic pair} for $T$, in the sense that
\[
\udens\{n \in \N : \|T^nx - T^ny\| \geq \eps\} = 1
\]
and
\[
\udens\{n \in \N : \|T^nx - T^ny\| < \tau\} = 1 \ \text{ for all } \tau > 0.
\]
By \cite[Theorem~12]{BerBonMulPer13}, $T$ is distributionally chaotic if and only if $T$ admits a {\em distributionally irregular vector},
that is, a vector $x \in X$ for which there exist $A,B \subset \N$ with $\udens(A) = \udens(B) = 1$ such that
\[
\lim_{n \in A} \|T^nx\| = 0 \ \ \ \text{ and } \ \ \ \lim_{n \in B} \|T^nx\| = \infty.
\]

\begin{theorem}\label{AE-T3}
Let $X = c_0(\Z)$ or $X = \ell_p(\Z)$ for some $p \in [1,\infty)$.
There exists a bounded sequence $w = (w_j)_{j \in \Z}$ of scalars, with $\inf_{j \in \Z} |w_j| > 0$, such that
the bilateral weighted backward shift $B_w : X \to X$ is topologically transitive and 
both $B_w$ and $B_w^{-1}$ are average positively expansive and distributionally chaotic.
\end{theorem}

Recall that topological transitivity is equivalent to {\em hypercyclicity} (i.e., the existence of a dense orbit) 
for operators on separable $F$-spaces.

The proof below is inspired by the construction of a certain unilateral weighted forward shift made in \cite[Theorem~25]{BerBonPerWu18}.

\begin{proof}
We define $w = (w_j)_{j \in \Z}$ as a sequence of blocks
\[
w = (\ldots B_3 A_3 B_2 A_2 B_1 A_1 I C_1 B_1 C_2 B_2 C_3 B_3 \ldots),
\]
where $I = (1,1)$, each $A_j$ has the form
\[
A_j = \Big(\underbrace{1,1,\ldots,1}_{2^{k_j} - 1 \text{ times}},\frac{j}{2(j+1)},
\underbrace{\frac{1}{2},\frac{1}{2},\ldots,\frac{1}{2}}_{2k_j - 1 \text{ times}},1,\underbrace{2,2,\ldots,2}_{2k_j \text{ times}}\Big),
\]
each $B_j$ has the form
\[
B_j = \Big(\underbrace{\frac{1}{2},\frac{1}{2},\ldots,\frac{1}{2}}_{r_j \text{ times}},
\underbrace{1,1,\ldots,1}_{i_j - 1 \text{ times}},\underbrace{2,2,\ldots,2}_{r_j \text{ times}}\Big),
\]
each $C_j$ has the form
\[
C_j = \Big(\underbrace{\frac{1}{2},\frac{1}{2},\ldots,\frac{1}{2}}_{2k_j \text{ times}},1,
\underbrace{2,2,\ldots,2}_{2k_j - 1 \text{ times}},\frac{2(j+1)}{j},\underbrace{1,1,\ldots,1}_{2^{k_j} - 1 \text{ times}}\Big),
\]
and the first number $1$ in $I$ lies at position $0$ in $w$.
The number $r_j$ is the smallest positive integer greater than or equal to $2\log_2(j+1)$ 
and the numbers $k_j$ and $i_j$ will be chosen inductively in the next paragraph.
Let $a_j = 4k_j + 2^{k_j}$ denote the length of the blocks $A_j$ and $C_j$, and let $b_j = 2r_j + i_j - 1$ denote the length of the block $B_j$.
We also define $s_j = a_1+b_1+a_2+b_2+\cdots+a_j$ and $t_j = s_j + b_j$ for $j \in \N$, and $t_0 = 0$.

We define $k_1 = 2$ and $i_1 = 2$. Hence,
\begin{align*}
A_1 &= \Big(1,1,1,\frac{1}{4},\frac{1}{2},\frac{1}{2},\frac{1}{2},1,2,2,2,2\Big), \\
B_1 &= \Big(\frac{1}{2},\frac{1}{2},1,2,2\Big), \\
C_1 &= \Big(\frac{1}{2},\frac{1}{2},\frac{1}{2},\frac{1}{2},1,2,2,2,4,1,1,1\Big).
\end{align*}
Suppose that $j \geq 2$ and that $k_1,\ldots,k_{j-1}$ and $i_1,\ldots,i_{j-1}$ have been chosen. Note that
\begin{align*}
\big(\|B_w^n e_{-1}\|\big)_{t_{j-1} < n \leq s_j} &= \big(\|B_w^{-n} e_1\|\big)_{t_{j-1} < n \leq s_j}\\
&= \Big(\frac{2}{j},\frac{2^2}{j},\ldots,\frac{2^{2k_j}}{j},\frac{2^{2k_j}}{j},
  \frac{2^{2k_j-1}}{j},\ldots,\frac{2}{j},\underbrace{\frac{1}{j+1},\frac{1}{j+1},\ldots,\frac{1}{j+1}}_{2^{k_j} \text{ times}}\Big)
\end{align*}
and
\begin{align*}
\big(\|B_w^n e_{-1}\|\big)_{s_j < n \leq t_j} &= \big(\|B_w^{-n} e_1\|\big)_{s_j < n \leq t_j}\\
&= \Big(\frac{2}{j+1},\ldots,\frac{2^{r_j-1}}{j+1},
  \underbrace{\frac{2^{r_j}}{j+1},\frac{2^{r_j}}{j+1},\ldots,\frac{2^{r_j}}{j+1}}_{i_j \text{ times}},
  \frac{2^{r_j-1}}{j+1},\ldots,\frac{2}{j+1},\frac{1}{j+1}\Big).
\end{align*}
Since
\begin{align*}
\frac{\card\{1 \leq n \leq s_j : \|B_w^n e_{-1}\| \leq \frac{1}{j+1}\}}{s_j} 
&= \frac{\card\{1 \leq n \leq s_j : \|B_w^{-n} e_1\| \leq \frac{1}{j+1}\}}{s_j}\\
&\geq \frac{2^{k_j}}{t_{j-1} + 4k_j + 2^{k_j}} \to 1 \ \text{ as } k_j \to \infty,
\end{align*}
we can choose $k_j > k_{j-1}$ large enough so that
\begin{align}\label{AE-T3-Eq1}
\frac{\card\{1 \leq n \leq s_j : \|B_w^n e_{-1}\| \leq \frac{1}{j+1}\}}{s_j} 
&= \frac{\card\{1 \leq n \leq s_j : \|B_w^{-n} e_1\| \leq \frac{1}{j+1}\}}{s_j}\\
&\geq 1 - \frac{1}{j}\cdot \nonumber
\end{align}
Let
\[
\alpha_j = \frac{1}{s_j} \sum_{n=1}^{s_j} \|B_w^n e_{-1}\| = \frac{1}{s_j} \sum_{n=1}^{s_j} \|B_w^{-n} e_1\|.
\]
Since
\[
\alpha_j = \frac{\sum_{n=1}^{t_{j-1}} \|B_w^n e_{-1}\| + \frac{2(2^{2k_j + 1}-2)}{j} + \frac{2^{k_j}}{j+1}}{t_{j-1} + 4k_j + 2^{k_j}}
  \geq \frac{2^{2k_j+2} - 4}{j(t_{j-1} + 4k_j + 2^{k_j})} \to \infty \ \text{ as } k_j \to \infty,
\]
we can assume that $k_j$ was chosen so large that
\begin{equation}\label{AE-T3-Eq2}
\frac{s_j \alpha_j}{s_j + 4 r_j} \geq j+1.
\end{equation}
On the other hand, since $2^{r_j} \geq (j+1)^2$, we have that
\begin{align*}
\frac{\card\{1 \leq n \leq t_j : \|B_w^n e_{-1}\| \geq j+1\}}{t_j}
&= \frac{\card\{1 \leq n \leq t_j : \|B_w^{-n} e_1\| \geq j+1\}}{t_j}\\
&\geq \frac{i_j}{s_j + 2r_j + i_j - 1} \to 1 \ \text{ as } i_j \to \infty.
\end{align*}
Therefore, we can choose $i_j > i_{j-1}$ so large that
\begin{align}\label{AE-T3-Eq3}
\frac{\card\{1 \leq n \leq t_j : \|B_w^n e_{-1}\| \geq j+1\}}{t_j}
&= \frac{\card\{1 \leq n \leq t_j : \|B_w^{-n} e_1\| \geq j+1\}}{t_j}\\
&\geq 1 - \frac{1}{j}\cdot \nonumber
\end{align}
By induction, this completes the construction of the sequences $(k_j)_{j \in \N}$ and $(i_j)_{j \in \N}$.

By (\ref{AE-T3-Eq1}),
\[
\udens\{n \in \N : \|B_w^n e_{-1}\| < \tau\} = \udens\{n \in \N : \|B_w^{-n} e_1\| < \tau\} = 1 \ \text{ for all } \tau > 0.
\]
By (\ref{AE-T3-Eq3}),
\[
\udens\{n \in \N : \|B_w^n e_{-1}\| \geq k\} = \udens\{n \in \N : \|B_w^{-n} e_1\| \geq k\} = 1 \ \text{ for all } k \in \N.
\]
Thus, $e_{-1}$ is a distributionally irregular vector for $B_w$ and $e_1$ is a distributionally irregular vector for $B_w^{-1}$,
showing that both $B_w$ and $B_w^{-1}$ are distributionally chaotic.

Now, we claim that
\begin{equation}\label{AE-T3-Eq4}
\frac{1}{n} \sum_{i=1}^n \|B_w^i e_{-1}\| = \frac{1}{n} \sum_{i=1}^n \|B_w^{-i} e_1\| \geq j+1 \ \ 
  \text{ whenever } t_{j-1} + 4k_j \leq n \leq t_j + 4k_{j+1}.
\end{equation}
In fact, fix $n \in [t_{j-1} + 4k_j, t_j + 4k_{j+1}]$ (in this paragraph we are considering intervals of integers).
It is enough to prove that
\begin{equation}\label{AE-T3-Eq5}
\frac{1}{n} \sum_{i=1}^n \|B_w^i e_{-1}\| \geq j+1.
\end{equation}
By (\ref{AE-T3-Eq2}),
\begin{equation}\label{AE-T3-Eq6}
\frac{1}{s_j + 4 r_j} \sum_{i=1}^{s_j} \|B_w^i e_{-1}\| \geq j+1.
\end{equation}
In particular, the average $\frac{1}{s_j} \sum_{i=1}^{s_j} \|B_w^i e_{-1}\|$ is greater than or equal to $j+1$.
This implies that (\ref{AE-T3-Eq5}) holds if $n \in [t_{j-1} + 4k_j, s_j]$, 
since $\|B_w^i e_{-1}\| = \frac{1}{j+1} < j+1$ for all $i \in [t_{j-1} + 4k_j + 1, s_j]$.
In the case $n$ belongs to the interval $[s_j, t_j + 4k_{j+1}]$, we observe that this interval can be divided into five subintervals, namely:
\begin{align*}
&[s_j, s_j + r_j -1], \ \ [s_j + r_j, t_j - r_j], \ \ [t_j - r_j + 1, t_j + r_j],\\
&[t_j + r_j +1, t_j + 4 k_{j+1} - r_j] \ \text{ and } [t_j + 4 k_{j+1} - r_j + 1, t_j + 4 k_{j+1}].
\end{align*}
For each $i$ in the second or fourth subinterval above, we have $\|B_w^i e_{-1}\| \geq j+1$.
On the other hand, the lengths of the other three subintervals add up to $4 r_j$.
Since we have a $4 r_j$ in the denominator in (\ref{AE-T3-Eq6}), 
we see that the inequality in (\ref{AE-T3-Eq6}) implies that (\ref{AE-T3-Eq5}) also holds in this case.
This completes the proof of our claim. It follows from (\ref{AE-T3-Eq4}) that
\[
\lim_{n \to \infty} \frac{1}{n} \sum_{j=1}^n \|B_w^j x\| = \lim_{n \to \infty} \frac{1}{n} \sum_{j=1}^n \|B_w^{-j} x\| = \infty
  \ \ \text{ for all } x \neq 0.
\]
In particular, both $B_w$ and $B_w^{-1}$ are average positively expansive.

Finally, in order to prove that $B_w$ is topologically transitive, we shall apply the characterization of hypercyclicity contained
in \cite[Example~4.15]{KGroAPer11}, namely: {\it $B_w$ is hypercyclic if and only if there is an increasing sequence $(n_j)_{j \in \N}$
of positive integers such that, for every $t \in \Z$,}
\[
w_{t - n_j + 1} \cdots w_t \to 0 \ \textit{ and } \ w_{t+1}^{-1} \cdots w_{t+n_j}^{-1} \to 0 \ \textit{ as } j \to \infty.
\]
For our $B_w$, we consider $n_j = t_{j-1} + 4k_j + 2^{k_j - 1}$ ($j \in \N$). Note that
\[
w_{-n} \cdots w_{-1} = \|B_w^n e_{-1}\| = \frac{1}{j+1} \ \ \text{ and } \ \ w_2^{-1} \cdots w_{n+1}^{-1} = \|B_w^{-n} e_1\| = \frac{1}{j+1}\,,
\]
whenever $n_j - 2^{k_j - 1} < n \leq n_j + 2^{k_j - 1}$ ($j \in \N$).
Let $t \in \Z$ be arbitrary and choose $j_0 \in \N$ such that $|t| < 2^{k_{j_0} - 1}$.
For any $j \geq j_0$,
\[
w_{t - n_j + 1} \cdots w_t = \left\{\begin{array}{cl}
  \frac{w_{-(n_j - t)} \cdots w_{-1} w_0 \cdots w_t}{w_{t - n_j}} = \frac{1}{j+1} \cdot \frac{w_0 \cdots w_t}{w_{t - n_j}} & \text{if } t \geq 0\\
  w_{-n_j} \cdots w_{-1} = \frac{1}{j+1} & \text{if } t = -1\,.\\
  \frac{w_{-(n_j - t)} \cdots w_{t} w_{t+1} \cdots w_{-1}}{w_{t - n_j} \cdot\, w_{t+1} \cdots w_{-1}} 
    = \frac{1}{j+1} \cdot \frac{1}{w_{t - n_j} \cdot\, w_{t+1} \cdots w_{-1}}& \text{if } t \leq -2
\end{array}\right.
\]
Hence, $w_{t - n_j + 1} \cdots w_t \to 0$ as $j \to \infty$.
The computations showing that $w_{t+1}^{-1} \cdots w_{t+n_j}^{-1} \to 0$ as $j \to \infty$ are similar and are left to the reader.
\end{proof}

\begin{remark}
As we saw in the previous theorem, an average expansive weighted shift can be topologically transitive.
However, it cannot be topologically mixing. 
In fact, let $B_w$ be a bilateral weighted backward shift on a Fr\'echet sequence space $X$ over $\Z$
in which the sequence $(e_j)_{j \in \Z}$ of canonical vectors is a basis.
By \cite[Theorem~4.13]{KGroAPer11}, $B_w$ is topologically mixing if and only if
\[
w_{-j+1} \cdots w_0\, e_{-j} \to 0 \ \ \ \text{ and } \ \ \ \frac{e_j}{w_1 \cdots w_j} \to 0 \ \ \text{ as } j \to \infty.
\]
By Theorem~\ref{AE-T2}, this implies that $B_w$ cannot be average expansive.
\end{remark}


\section{Uniformly expansive weighted shifts on K\"othe sequence spaces}\label{S-UE-K}

Our next result characterizes the uniformly expansive bilateral weighted forward shifts on K\"othe sequence spaces.
For this purpose, we adopt the convention that the infimum of a family indexed by the empty set is equal to $\infty$.

\begin{theorem}\label{UEWS-T1}
Consider a K\"othe sequence space $X = \lambda_p(A,\Z)$, where $A = (a_{j,k})$ is a K\"othe matrix on $\Z$ 
and $p \in \{0\} \cup [1,\infty)$. For each $k \in \N$, define
\[
I_k = \{j \in \Z : a_{j,k} \neq 0\}.
\]
Let $w = (w_j)_{j \in \Z}$ be a sequence of nonzero scalars such that the bilateral weighted forward shift
\[
F_w : (x_j)_{j \in \Z} \in X \mapsto (w_{j-1}x_{j-1})_{j \in \Z} \in X
\]
is an invertible operator on $X$ (i.e., conditions (\ref{Fw-Defined}) and (\ref{Fw-Invertible}) hold). 
Then $F_w$ is uniformly expansive if and only if one of the following properties holds:
\begin{itemize}
\item [\rm (A)] For each $k \in \N$, there exists $\ell \in \N$ such that
    \[
    \lim_{n \to \infty} \Big(\inf_{j \in I_k} \frac{a_{j+n,\ell} |w_j \cdots w_{j+n-1}|}{a_{j,k}}\Big) = \infty.
    \]
\item [\rm (B)] For each $k \in \N$, there exists $\ell \in \N$ such that
    \[
    \lim_{n \to \infty} \Big(\inf_{j \in I_k} \frac{a_{j-n,\ell}}{a_{j,k} |w_{j-n} \cdots w_{j-1}|}\Big) = \infty.
    \]
\item [\rm (C)] For each $k \in \N$, there exists $\ell \in \N$ such that
    \[
    \lim_{n \to \infty} \Big(\inf_{j \in I_k \cap \N} \frac{a_{j+n,\ell} |w_j \cdots w_{j+n-1}|}{a_{j,k}}\Big) = \infty
    \]
    and
    \[
    \lim_{n \to \infty} \Big(\inf_{j \in I_k \cap (-\N)} \frac{a_{j-n,\ell}}{a_{j,k} |w_{j-n} \cdots w_{j-1}|}\Big) = \infty.
    \]
\end{itemize}
\end{theorem}

\begin{proof}
($\Rightarrow$): By the definition of uniform expansivity, for each $k \in \N$, there exist $\ell_k \in \N$ 
and subsets $S_k^+$ and $S_k^-$ of $S_{\|\cdot\|_k}$ such that $S_{\|\cdot\|_k} = S_k^+ \cup S_k^-$,
\[
\|F_w^n(x)\|_{\ell_k} \to \infty \text{ uniformly on } S_k^+ \text{ as } n \to \infty
\]
and
\[
\|F_w^{-n}(x)\|_{\ell_k} \to \infty \text{ uniformly on } S_k^- \text{ as } n \to \infty.
\]
For each $n \in \N$, let
\[
C_{k,n} = \inf_{x \in S_k^+} \|F_w^n(x)\|_{\ell_k} \ \ \text{ and } \ \ D_{k,n} = \inf_{x \in S_k^-} \|F_w^{-n}(x)\|_{\ell_k}.
\]
Then $C_{k,n} \to \infty$ and $D_{k,n} \to \infty$ as $n \to \infty$. Consider the sets
\[
I_k^+ = \Big\{j \in I_k : \frac{e_j}{a_{j,k}} \in S_k^+\Big\} \ \ \text{ and } \ \ I_k^- = \Big\{j \in I_k : \frac{e_j}{a_{j,k}} \in S_k^-\Big\}.
\]
Clearly, $I_k = I_k^+ \cup I_k^-$. It follows from the definitions that, for each $n \in \N$,
\[
\inf_{j \in I_k^+} \frac{a_{j+n,\ell_k} |w_j \cdots w_{j+n-1}|}{a_{j,k}} 
  = \inf_{j \in I_k^+} \Big\|F_w^n\Big(\frac{e_j}{a_{j,k}}\Big)\Big\|_{\ell_k} \geq C_{k,n}
\]
and
\[
\inf_{j \in I_k^-} \frac{a_{j-n,\ell_k}}{a_{j,k} |w_{j-n} \cdots w_{j-1}|} 
  = \inf_{j \in I_k^-} \Big\|F_w^{-n}\Big(\frac{e_j}{a_{j,k}}\Big)\Big\|_{\ell_k} \geq D_{k,n}.
\]
Therefore, 
\begin{equation}\label{UEWS-F1}
\lim_{n \to \infty} \Big(\inf_{j \in I_k^+} \frac{a_{j+n,\ell_k} |w_j \cdots w_{j+n-1}|}{a_{j,k}}\Big) = \infty\, \text{ and }
\lim_{n \to \infty} \Big(\inf_{j \in I_k^-} \frac{a_{j-n,\ell_k}}{a_{j,k} |w_{j-n} \cdots w_{j-1}|}\Big) = \infty.
\end{equation}
If $I_k^- = \emptyset$ for all $k \in \N$, then (A) holds, and if $I_k^+ = \emptyset$ for all $k \in \N$, then (B) holds.
Suppose that there exist $r,s \in \N$ such that $I_r^- \neq \emptyset$ and $I_s^+ \neq \emptyset$.
Then (\ref{UEWS-F1}) implies that there is $t \in \N$ such that $I_k = \Z$ for all $k \geq t$.
Let $k_0 = \max\{t,\ell_r,\ell_s\}$. We claim that
\begin{equation}\label{UEWS-F2}
I_k^+ \text{ is bounded below and } I_k^- \text{ is bounded above for all } k \geq k_0.
\end{equation}
Indeed, take $j_1 \in I_r^-$ and $j_2 \in I_s^+$. Then,
\[
\lim_{n \to \infty} \frac{a_{j_1-n,\ell_r}}{a_{j_1,r} |w_{j_1-n} \cdots w_{j_1-1}|} = \infty \ \ \text{ and } \
\lim_{n \to \infty} \frac{a_{j_2+n,\ell_s} |w_{j_2} \cdots w_{j_2+n-1}|}{a_{j_2,s}} = \infty,
\]
which implies that
\begin{equation}\label{UEWS-F3}
\lim_{n \to \infty} \frac{a_{-n,\ell_r}}{|w_{-n} \cdots w_{-1}|} = \infty \ \ \text{ and } \
\lim_{n \to \infty} a_{n,\ell_s} |w_0 \cdots w_{n-1}| = \infty.
\end{equation}
Assume $k \geq k_0$. If $n \in \N$ and $-n \in I_k^+$, then
\[
\frac{a_{0,\ell_k} |w_{-n} \cdots w_{-1}|}{a_{-n,\ell_r}}   \geq \frac{a_{0,\ell_k} |w_{-n} \cdots w_{-1}|}{a_{-n,k}} \geq C_{k,n}.
\]
Hence, the first equality in (\ref{UEWS-F3}) implies that $I_k^+$ must be bounded below.
Similarly, if $n \in \N$ and $n \in I_k^-$, then
\[
\frac{a_{0,\ell_k}}{a_{n,\ell_s} |w_0 \cdots w_{n-1}|} \geq \frac{a_{0,\ell_k}}{a_{n,k} |w_0 \cdots w_{n-1}|} \geq D_{k,n}.
\]
Thus, the second equality in (\ref{UEWS-F3}) implies that $I_k^-$ must be bounded above, completing the proof of our claim.

For $k \geq k_0$, (\ref{UEWS-F2}) implies that the sets $I_k^+$ and $I_k^-$ are nonempty (because $I_k = \Z$) and, consequently, 
the limits in (\ref{UEWS-F1}) do not change if we remove (resp.\ add) a finite number of integers from (resp.\ to) these sets.
Thus, it follows from (\ref{UEWS-F2}) that (\ref{UEWS-F1}) is equivalent to the limits in (C).

For $k < k_0$, we have that $I_k \subset I_{k_0}$ and, therefore, the limits in (C) hold with $\ell = \ell_{k_0}$.

\smallskip\noindent
($\Leftarrow$): We first assume that (C) holds. Given $k \in \N$, let $\ell \in \N$ be as in (C). Let 
\[
I_k^+ = I_k \cap \N_0 \ \ \text{ and } \ \ I_k^- = I_k \cap (-\N).
\]
For each $x = (x_j)_{j \in \Z} \in X$, we define $x^+ = (x^+_j)_{j \in \Z} \in X$ and $x^- = (x^-_j)_{j \in \Z} \in X$ by
\[
x^+_j = \left\{\begin{array}{cl} x_j & \text{if } j \in I_k^+\\ 0 & \text{otherwise} \end{array}\right. \ \ \text{ and } \ \ \
x^-_j = \left\{\begin{array}{cl} x_j & \text{if } j \in I_k^-\\ 0 & \text{otherwise} \end{array}\right..
\]
Let 
\[
S_k^+ = \{x \in S_{\|\cdot\|_k} : \|x^+\|_k \geq 1/2\} \ \ \text{ and } \ \ S_k^- = \{x \in S_{\|\cdot\|_k} : \|x^-\|_k \geq 1/2\}.
\]
Clearly, $S_{\|\cdot\|_k} = S_k^+ \cup S_k^-$. For each $n \in \N$, let
\[
C_n = \inf_{j \in I_k^+} \frac{a_{j+n,\ell} |w_j \cdots w_{j+n-1}|}{a_{j,k}} \ \ \text{ and } \ \
D_n = \inf_{j \in I_k^-} \frac{a_{j-n,\ell}}{a_{j,k} |w_{j-n} \cdots w_{j-1}|}\cdot
\]
By our choice of $\ell$, $C_n \to \infty$ and $D_n \to \infty$ as $n \to \infty$. Take $y \in S_k^+$. If $p \in [1,\infty)$, then
\[
\|F_w^n(y)\|_\ell^p \geq \|F_w^n(y^+)\|_\ell^p = \sum_{j \in I_k^+ + n} |a_{j,\ell}\, w_{j-n} \cdots w_{j-1}\, y^+_{j-n}|^p
  \geq C_n^p \, \|y^+\|_k^p \geq \Big(\frac{C_n}{2}\Big)^p.
\]
If $p = 0$, then
\[
\|F_w^n(y)\|_\ell \geq \|F_w^n(y^+)\|_\ell = \sup_{j \in I_k^+ + n} |a_{j,\ell}\, w_{j-n} \cdots w_{j-1}\, y^+_{j-n}|
  \geq C_n \, \|y^+\|_k \geq \frac{C_n}{2}\cdot
\]
In any case, it follows that 
\[
\|F_w^n(x)\|_\ell \to \infty \ \text{ uniformly on } S_k^+ \text{ as } n \to \infty.
\]
Now, take $z \in S_k^-$. If $p \in [1,\infty)$, then
\[
\|F_w^{-n}(z)\|_\ell^p \geq \|F_w^{-n}(z^-)\|_\ell^p 
  = \sum_{j \in I_k^- - n} \Big|\frac{a_{j,\ell}\, z^-_{j+n}}{w_j \cdots w_{j+n-1}}\Big|^p
  \geq D_n^p \, \|z^-\|_k^p \geq \Big(\frac{D_n}{2}\Big)^p.
\]
If $p = 0$, then
\[
\|F_w^{-n}(z)\|_\ell \geq \|F_w^{-n}(z^-)\|_\ell 
  = \sup_{j \in I_k^- - n} \Big|\frac{a_{j,\ell}\, z^-_{j+n}}{w_j \cdots w_{j+n-1}}\Big|
  \geq D_n \, \|z^-\|_k \geq \frac{D_n}{2}\cdot
\]
In any case, we see that 
\[
\|F_w^{-n}(x)\|_\ell \to \infty \ \text{ uniformly on } S_k^- \text{ as } n \to \infty.
\]

If (A) holds, then similar computations show that
\[
\|F_w^n(x)\|_\ell \to \infty \ \text{ uniformly on } S_{\|\cdot\|_k} \text{ as } n \to \infty.
\]

Analogously, if (B) holds, then
\[
\|F_w^{-n}(x)\|_\ell \to \infty \ \text{ uniformly on } S_{\|\cdot\|_k} \text{ as } n \to \infty.
\]

This completes the proof that $F_w$ is uniformly expansive.
\end{proof}

\begin{remark}\label{UEWS-R1}
It is implicit in the proof of Theorem~\ref{UEWS-T1} that properties (A), (B) and (C) are mutually exclusive.
Note also that properties (A), (B) and (C) imply that, for every sufficiently large $k$,
\[
I_k \supset \N, \ \ \ \ I_k \supset -\N \ \ \ \text{ and } \ \ \ I_k = \Z,
\]
respectively. Nevertheless, it may happen that (A) holds and $I_k \neq \Z$ for all $k$.
As an example, consider $a_{j,k} = 1$ if $j > -k$, $a_{j,k} = 0$ if $j \leq -k$ and $w_j = 2$ ($j \in \Z$, $k \in \N$).
Since conditions (\ref{Fw-Defined}) and (\ref{Fw-Invertible}) are true, $F_w$ is an invertible operator on $\lambda_p(A,\Z)$,
which is uniformly expansive because property (A) holds.
A similar example can be defined in the case of property (B).
\end{remark}

\begin{example}\label{UEWS-E1}
If $a_{j,k} = 1$ for all $j \in \Z$ and $k \in \N$, then
\[
\lambda_0(A,\Z) = c_0(\Z) \ \ \text{ and } \ \ \lambda_p(A,\Z) = \ell^p(\Z) \text{ for } p \in [1,\infty).
\]
In this case, Theorem~\ref{UEWS-T1} recovers \cite[Theorem~E(2)]{BerCirDarMesPuj18}.
\end{example}

Let $T$ be a uniformly expansive operator on a normed space $X$. The proof of \cite[Proposition~19(d)]{BerCirDarMesPuj18} 
essentially shows that there exist constants $c > 0$ and $t > 1$ such that we can write $S_X = S^+ \cup S^-$ with
\[
\|T^n y\| \geq c\, t^n \ \text{ and } \ \|T^{-n} z\| \geq c\, t^n \ \text{ for all } n \in \N, \ y \in S^+ \text{ and } z \in S^-.
\]
Hence, the trajectory $(T^n x)_{n \in \Z}$ of any nonzero vector $x \in X$ grows exponentially fast as $n \to \infty$ or as $n \to -\infty$.
We shall now show that things may be different if we go beyond the normed space setting.
More precisely, we will give an example of a uniformly expansive operator on the Fr\'echet space $s(\Z)$ of rapidly decreasing sequences 
whose trajectories have only polynomial (not exponential) growth.

\begin{example}\label{UEWS-E2}
Let $a_{j,k} = (|j| + 1)^k$ for all $j \in \Z$ and $k \in \N$. Then
\[
\lambda_1(A,\Z) = s(\Z),
\]
the {\em space of rapidly decreasing sequences} (on $\Z$). The bilateral forward shift
\[
F : (x_j)_{j \in \Z} \in s(\Z) \mapsto (x_{j-1})_{j \in \Z} \in s(\Z)
\]
is an invertible operator on $s(\Z)$, which is uniformly expansive, since property (C) holds:
\[
\lim_{n \to \infty} \inf_{j \in \N} \frac{a_{j+n,k+1}}{a_{j,k}} 
  = \lim_{n \to \infty} \inf_{j \in \N} \frac{(j+n+1)^{k+1}}{(j+1)^k} \geq \lim_{n \to \infty} n = \infty
\]
and
\[
\lim_{n \to \infty} \inf_{j \in -\N} \frac{a_{j-n,k+1}}{a_{j,k}} 
  = \lim_{n \to \infty} \inf_{j \in -\N} \frac{(n-j+1)^{k+1}}{(-j+1)^k} \geq \lim_{n \to \infty} n = \infty.
\]
For any $k \in \N$, $x = (x_j)_{j \in \Z} \in s(\Z)$ and $n \in \Z$,
\begin{align*}
\|F^n(x)\|_k &= \sum_{j=-\infty}^\infty (|j|+1)^k |x_{j-n}| 
    = \sum_{j=-\infty}^\infty \Big(\frac{|j+n|+1}{|j|+1}\Big)^k (|j|+1)^k |x_j|\\
  &\leq (|n| + 1)^k \sum_{j=-\infty}^\infty (|j|+1)^k |x_j| = (|n| + 1)^k \|x\|_k.
\end{align*}
Thus, the forward trajectory $(F^n(x))_{n \in \N}$ and the backward trajectory $(F^{-n}(x))_{n \in \N}$ 
can grow at most in order $O(n^k)$ with respect to $\|\cdot\|_k$.
\end{example}

\begin{remark}\label{UEWS-R2}
Let $X$ be a locally convex space over $\K$ whose topology is induced by a family $(\|\cdot\|_\alpha)_{\alpha \in I}$ of seminorms.
Recall \cite{BCDFP} that an operator $T \in L(X)$ is said to be {\em positively expansive} if the following condition holds:
\begin{itemize}
\item [(PE)] For each $x \in X \backslash \{0\}$, there exists $\alpha \in I$ such that $\sup_{n \in \N} \|T^n x\|_\alpha = \infty$.
\end{itemize}
Let us say that $T$ is {\em uniformly positively expansive} if:
\begin{itemize}
\item [(UPE)] For each $\alpha \in I$, there exists $\beta \in I$ such that
\[
\|T^n x\|_\beta \to \infty \text{ uniformly on } S_{\|\cdot\|_\alpha} \text{ as } n \to \infty.
\]
\end{itemize}
Clearly, the notions of positive expansivity, average positive expansivity and uniform positive expansivity do not depend on the choice 
of the family of seminorms inducing the topology of $X$. Moreover, as in Proposition~\ref{AE-P1}, we have that
\begin{align*}
T \text{ is uniformly positively expansive } \ &\Longrightarrow \ \ T \text{ is average positively expansive }\\
 &\Longrightarrow \ \ T \text{ is positively expansive}.
\end{align*}
Similar (but simpler) arguments to those used in the proof of Theorem~\ref{UEWS-T1} show that:
\begin{quote}
{\it A bilateral weighted forward shift $F_w$ on a K\"othe sequence space $\lambda_p(A,\Z)$
is uniformly positively expansive exactly when property (A) holds}.
\end{quote}
For unilateral weighted forward shifts, the characterization is analogous:
{\it A unilateral weighted forward shift $F_w$ on a K\"othe sequence space $\lambda_p(A,\N)$
is uniformly positively expansive if and only if, for each $k \in \N$, there exists $\ell \in \N$ such that
\[
\lim_{n \to \infty} \Big(\inf_{j \in I_k} \frac{a_{j+n,\ell} |w_j \cdots w_{j+n-1}|}{a_{j,k}}\Big) = \infty,
\]
where $I_k = \{j \in \N : a_{j,k} \neq 0\}$.}
\end{remark}

\begin{remark}\label{UEWS-R3}
Suppose that a bilateral weighted backward shift $B_w$ is an invertible operator on a K\"othe sequence space $\lambda_p(A,\Z)$.
Since $B_w$ is uniformly expansive if and only if so is $B_w^{-1}$, and since 
\[
B_w^{-1} = F_{w'}, \ \text{ where } w' = \Big(\frac{1}{w_{j+1}}\Big)_{j \in \Z},
\]
it follows from Theorem~\ref{UEWS-T1} that $B_w$ is uniformly expansive if and only if 
one of the following (mutually exclusive) properties holds:
\begin{itemize}
\item [\rm (a)] For each $k \in \N$, there exists $\ell \in \N$ such that
    \[
    \lim_{n \to \infty} \Big(\inf_{j \in I_k} \frac{a_{j-n,\ell} |w_{j-n+1} \cdots w_j|}{a_{j,k}}\Big) = \infty.
    \]
\item [\rm (b)] For each $k \in \N$, there exists $\ell \in \N$ such that
    \[
    \lim_{n \to \infty} \Big(\inf_{j \in I_k} \frac{a_{j+n,\ell}}{a_{j,k} |w_{j+1} \cdots w_{j+n}|}\Big) = \infty.
    \]
\item [\rm (c)] For each $k \in \N$, there exists $\ell \in \N$ such that
    \[
    \lim_{n \to \infty} \Big(\inf_{j \in I_k \cap (-\N)} \frac{a_{j-n,\ell} |w_{j-n+1} \cdots w_j|}{a_{j,k}}\Big) = \infty
    \]
    and
    \[
    \lim_{n \to \infty} \Big(\inf_{j \in I_k \cap \N} \frac{a_{j+n,\ell}}{a_{j,k} |w_{j+1} \cdots w_{j+n}|}\Big) = \infty.
    \]
\end{itemize}
Moreover, (a) is exactly the case in which $B_w$ is uniformly positively expansive.
On the other hand, it is obvious that unilateral weighted backward shifts are never positively expansive.
\end{remark}


\section{Generalities on the various notions of expansivity for operators}\label{S-Gen}

In the sequel, $\cP$ denotes any of the following properties about invertible operators:
expansivity, average expansivity or uniform expansivity.

Recall that a family $(\|\cdot\|_\alpha)_{\alpha \in I}$ of seminorms on a vector space $X$ is said to be {\em directed} if
for every $\alpha,\beta \in I$, there exists $\gamma \in I$ such that
$\|\cdot\|_\alpha \leq \|\cdot\|_\gamma$ and $\|\cdot\|_\beta \leq \|\cdot\|_\gamma$.

\begin{proposition}\label{Gen-P1}
Let $X$ be a locally convex space. For every $T \in GL(X)$, we have that:
\begin{itemize}
\item [\rm (a)] $T$ has property $\cP$ $\Longleftrightarrow$ $T^{-1}$ has property $\cP$.
\item [\rm (b)] $T$ has property $\cP$ $\Longleftrightarrow$ $\lambda T$ has property $\cP$ for some (or all) $\lambda \in S_\K$.
\item [\rm (c)] $T$ has property $\cP$ $\Longleftrightarrow$ $T^k$ has property $\cP$ for some (or all) $k \in \Z \backslash \{0\}$.
\end{itemize}
\end{proposition}

\begin{proof}
Let $(\|\cdot\|_\alpha)_{\alpha \in I}$ be a directed family of seminorms that induces the topology of $X$.

Since (a) and (b) follow trivially from the definitions, let us prove (c).
In view of (a), we can replace $\Z \backslash \{0\}$ by $\N$ in (c). So, fix $k \in \N$.

If $T$ is uniformly expansive, it is clear that $T^k$ is also uniformly expansive.
Conversely, suppose that $T^k$ is uniformly expansive.
Given $\alpha \in I$, there exists $\beta \in I$ such that we can write $S_{\|\cdot\|_\alpha} = S^+_\alpha \cup S^-_\alpha$, where
\begin{align*}
&\|(T^k)^n x\|_\beta \to \infty \text{ uniformly on } S^+_\alpha \text{ as } n \to \infty,\\
&\|(T^k)^{-n} x\|_\beta \to \infty \text{ uniformly on } S^-_\alpha \text{ as } n \to \infty.
\end{align*}
Since the family $(T^r)_{|r| < k}$ is equicontinuous, there exist $C \in (0,\infty)$ and $\gamma \in I$ such that
\begin{equation}\label{Gen-Eq1}
\|T^r x\|_\beta \leq C \|x\|_\gamma \ \text{ whenever } x \in X \text{ and } |r| < k.
\end{equation}
For each $n \in \N$, write $n = q_n k + r_n$, where $q_n \in \N_0$ and $r_n \in \{0,\ldots,k-1\}$. By (\ref{Gen-Eq1}),
\begin{align*}
&\|T^n x\|_\gamma = \|T^{r_n} (T^k)^{q_n} x\|_\gamma \geq \frac{1}{C}\, \|(T^k)^{q_n} x\|_\beta \to \infty 
  \text{ uniformly on } S^+_\alpha \text{ as } n \to \infty,\\
&\|T^{-n} x\|_\gamma = \|T^{-r_n} (T^k)^{-q_n} x\|_\gamma \geq \frac{1}{C}\, \|(T^k)^{-q_n} x\|_\beta \to \infty 
  \text{ uniformly on } S^-_\alpha \text{ as } n \to \infty.
\end{align*}
Thus, $T$ is also uniformly expansive.

Suppose that $T^k$ is not average expansive, that is, there exists $x \in X \backslash \{0\}$ such that
\[
D_\beta = \sup_{n \in \N} \Big(\frac{1}{2n+1} \sum_{j=-n}^n \|(T^k)^j x\|_\beta \Big) < \infty \ \text{ for all } \beta \in I.
\]
Fix $\beta \in I$ and let $C$ and $\gamma$ be as in (\ref{Gen-Eq1}). 
Write each $n \in \N$ as in the previous paragraph. 
It follows from (\ref{Gen-Eq1}) that
\begin{align*}
\frac{1}{2n+1} \sum_{j=-n}^n \|T^j x\|_\beta &= \frac{1}{2n+1} \sum_{j=-q_n k - r_n}^{q_n k + r_n} \|T^j x\|_\beta
  \leq \frac{2kC}{2n+1} \sum_{i=-q_n}^{q_n} \|(T^k)^i x\|_\gamma\\
&\leq \frac{2k (2q_n + 1)C D_\gamma}{2n+1} = \frac{(4(n - r_n) + 2k)C D_\gamma}{2n+1} \to 2CD_\gamma \text{ as } n \to \infty.
\end{align*}
This proves that $T$ is not average expansive.
Conversely, if $T^k$ is average expansive, it follows easily from the inequality
\[
\frac{1}{2kn+1} \sum_{j=-kn}^{kn} \|T^j x\|_\beta \geq \frac{2n+1}{2kn+1} \cdot \frac{1}{2n+1} \sum_{i=-n}^n \|(T^k)^i x\|_\beta
\]
that $T$ is also average expansive.

The case where $\cP$ denotes the expansivity property is left to the reader.
\end{proof}

\begin{proposition}\label{Gen-P2}
Let $X$ be a locally convex space. If $T \in GL(X)$ has property $\cP$ and $Y$ is a subspace of $X$ with $T(Y) = Y$,
then the restriction $T|_Y \in GL(Y)$ also has property $\cP$.
\end{proposition}

\begin{proof}
It follows easily from the definitions.
\end{proof}

\begin{proposition}\label{Gen-P3}
Let $X$ and $Y$ be locally convex spaces, $T \in GL(X)$ and $S \in GL(Y)$.
If there exists a linear homeomorphism $\phi : X \to Y$ such that $\phi \circ T = S \circ \phi$, then
\[
T \text{ has property } \cP \ \ \Longleftrightarrow \ \ S \text{ has property } \cP.
\]
\end{proposition}

\begin{proof}
Let $(\|\cdot\|_\alpha)_{\alpha \in I}$ (resp.\ $(|\cdot|_\beta)_{\beta \in J}$) be a directed family of seminorms 
that induces the topology of $X$ (resp.\ $Y$).
Without loss of generality, we may assume that the sets $I$ and $J$ are disjoint.
By symmetry, it is enough to prove the implication $\Rightarrow$.
We shall give the proof in the case $\cP$ means uniform expansivity. 
The other cases are simpler and are left to the reader.
So, suppose that $T$ is uniformly expansive.
Take $\beta \in J$. By the continuity of $\phi$, there exist $C \in (0,\infty)$ and $\alpha \in I$ such that
\begin{equation}\label{Gen-Eq2}
|\phi(x)|_\beta \leq C \|x\|_\alpha \ \text{ for all } x \in X.
\end{equation}
By the definition of uniform expansivity, there exists $\alpha' \in I$ such that 
we can write $S_{\|\cdot\|_\alpha} = S^+_\alpha \cup S^-_\alpha$, where
\begin{align}
&\|T^n x\|_{\alpha'} \to \infty \text{ uniformly on } S^+_\alpha \text{ as } n \to \infty, \label{Gen-Eq3} \\
&\|T^{-n} x\|_{\alpha'} \to \infty \text{ uniformly on } S^-_\alpha \text{ as } n \to \infty. \label{Gen-Eq4}
\end{align}
By the continuity of $\phi^{-1}$, there exist $C' \in (0,\infty)$ and $\beta' \in J$ such that
\begin{equation}\label{Gen-Eq5}
\|\phi^{-1}(y)\|_{\alpha'} \leq C' |y|_{\beta'} \ \text{ for all } y \in Y.
\end{equation}
It follows from (\ref{Gen-Eq2}) that
\begin{equation}\label{Gen-Eq6}
\|\phi^{-1}(y)\|_\alpha \geq \frac{1}{C} \ \text{ whenever } y \in S_{|\cdot|_\beta}.
\end{equation}
Thus, we can define the sets
\[
S^+_\beta = \Big\{y \in S_{|\cdot|_\beta} : \frac{\phi^{-1}(y)}{\|\phi^{-1}(y)\|_\alpha} \in S^+_\alpha\Big\} \ \text{ and } \
S^-_\beta = \Big\{y \in S_{|\cdot|_\beta} : \frac{\phi^{-1}(y)}{\|\phi^{-1}(y)\|_\alpha} \in S^-_\alpha\Big\}.
\]
Clearly, $S_{|\cdot|_\beta} = S^+_\beta \cup S^-_\beta$. By (\ref{Gen-Eq5}), (\ref{Gen-Eq6}) and (\ref{Gen-Eq3}),
\begin{align*}
|S^n y|_{\beta'} &\geq \frac{1}{C'} \|\phi^{-1}(S^n y)\|_{\alpha'} = \frac{1}{C'} \|T^n(\phi^{-1}(y))\|_{\alpha'}\\
  &\geq \frac{1}{C C'} \Big\|T^n\Big(\frac{\phi^{-1}(y)}{\|\phi^{-1}(y)\|_\alpha}\Big)\Big\|_{\alpha'}
  \to \infty \text{ uniformly on } S^+_\beta \text{ as } n \to \infty.
\end{align*}
Similarly, by (\ref{Gen-Eq5}), (\ref{Gen-Eq6}) and (\ref{Gen-Eq4}),
\begin{align*}
|S^{-n} y|_{\beta'} &\geq \frac{1}{C'} \|\phi^{-1}(S^{-n} y)\|_{\alpha'} = \frac{1}{C'} \|T^{-n}(\phi^{-1}(y))\|_{\alpha'}\\
  &\geq \frac{1}{C C'} \Big\|T^{-n}\Big(\frac{\phi^{-1}(y)}{\|\phi^{-1}(y)\|_\alpha}\Big)\Big\|_{\alpha'}
  \to \infty \text{ uniformly on } S^-_\beta \text{ as } n \to \infty.
\end{align*}
This proves that $S$ is uniformly expansive.
\end{proof}

\begin{proposition}\label{Gen-P4}
Let $X = \prod_{\ell \in L} X_\ell$ be the product of a family $(X_\ell)_{\ell \in L}$ of locally convex spaces and 
let $Y = \bigoplus_{\ell \in L} X_\ell$ be the external direct sum of this family 
(i.e., the set of all $(x_\ell)_{\ell \in L} \in X$ such that $x_\ell = 0$ except for a finite number of indices)
regarded as a subspace of $X$ (i.e., endowed with the induced topology).
Let $T : X \to X$ be the product operator
\[
T((x_\ell)_{\ell \in L}) = (T_\ell x_\ell)_{\ell \in L},
\]
where $T_\ell \in GL(X_\ell)$ for each $\ell \in X_\ell$. Then, the following assertions are equivalent:
\begin{itemize}
\item [\rm (i)] $T$ has property $\cP$;
\item [\rm (ii)] $T|_Y$ has property $\cP$;
\item [\rm (iii)] $T_\ell$ has property $\cP$ for all $\ell \in L$.
\end{itemize}
\end{proposition}

\begin{proof}
For each $\ell \in L$, let $(\|\cdot\|_\alpha)_{\alpha \in I_\ell}$ be a directed family of seminorms that induces the topology of $X_\ell$.
To avoid any ambiguity with the notation $\|\cdot\|_\alpha$, we assume that the sets $I_\ell$ are pairwise disjoint.
Recall that the product topology on $X$ is induced by the family of seminorms given by
\[
\vertiii{(x_\ell)_{\ell \in L}}_{\alpha_1,\ldots,\alpha_k} = \max\{\|x_{\ell_1}\|_{\alpha_1},\ldots,\|x_{\ell_k}\|_{\alpha_k}\},
\]
for $k \in \N$, $\ell_1,\ldots,\ell_k \in L$ and $\alpha_1 \in I_{\ell_1},\ldots,\alpha_k \in I_{\ell_k}$.

\smallskip\noindent
(i) $\Rightarrow$ (ii): It follows from Proposition~\ref{Gen-P2}.

\smallskip\noindent
(ii) $\Rightarrow$ (iii): For each $\ell \in L$, the {\em canonical embedding} $\phi_\ell : X_\ell \to Y$ is the function that maps each 
$z \in X_\ell$ to the family $(x_j)_{j \in L} \in Y$ whose $\ell^\text{th}$ coordinate is $z$ and the other coordinates are $0$.
Consider the subspace $\wt{X}_\ell = \phi_\ell(X_\ell)$ of $Y$.
Clearly, $\phi_\ell : X_\ell \to \wt{X}_\ell$ is a linear homeomorphism, 
\[
T(\wt{X}_\ell) = \wt{X}_\ell \ \ \text{ and } \ \ \phi_\ell \circ T_\ell = T|_{\wt{X}_\ell} \circ \phi_\ell.
\]
Hence, (iii) follows from Propositions~\ref{Gen-P2} and~\ref{Gen-P3}.

\smallskip\noindent
(iii) $\Rightarrow$ (i): We shall give the proof in the case $\cP$ means uniform expansivity. 
The other cases are simpler and are left to the reader. So, suppose that $T_\ell$ is uniformly expansive for all $\ell \in L$.
Consider a seminorm $\vertiii{\cdot}_{\alpha_1,\ldots,\alpha_k}$ of $X$, where
$k \in \N$, $\ell_1,\ldots,\ell_k \in L$ and $\alpha_1 \in I_{\ell_1},\ldots,\alpha_k \in I_{\ell_k}$.
For each $i \in \{1,\ldots,k\}$, the uniform expansivity of $T_{\ell_i}$ gives us an index $\beta_i \in I_{\ell_i}$ such that 
we can write $S_{\|\cdot\|_{\alpha_i}} = S^{+}_{\alpha_i} \cup S^{-}_{\alpha_i}$, where
\[
C^{+}_{i,n} = \inf_{z \in S^{+}_{\alpha_i}} \|(T_{\ell_i})^n z\|_{\beta_i} \to \infty \ \text{ and } \ 
C^{-}_{i,n} = \inf_{z \in S^{-}_{\alpha_i}} \|(T_{\ell_i})^{-n} z\|_{\beta_i} \to \infty \ \text{ as } n \to \infty.
\]
Let
\[
Z_i = \{x \in S_{\vertiii{\cdot}_{\alpha_1,\ldots,\alpha_k}} : \|x_{\ell_i}\|_{\alpha_i} = 1\}.
\]
Since $x \in S_{\vertiii{\cdot}_{\alpha_1,\ldots,\alpha_k}}$ implies $\|x_{\ell_i}\|_{\alpha_i} = 1$ for some $i \in \{1,\ldots,k\}$, we have that
\[
S_{\vertiii{\cdot}_{\alpha_1,\ldots,\alpha_k}} = Z_1 \cup \ldots \cup Z_k.
\]
Let $Z_i^+ = \{x \in Z_i : x_{\ell_i} \in S^{+}_{\alpha_i}\}$ and $Z_i^- = \{x \in Z_i : x_{\ell_i} \in S^{-}_{\alpha_i}\}$.
Clearly, $Z_i = Z_i^+ \cup Z_i^-$. Now, define
\[
S^{+}_{\alpha_1,\ldots,\alpha_k} = Z_1^+ \cup \ldots \cup Z_k^+ \ \ \text{ and } \ \
S^{-}_{\alpha_1,\ldots,\alpha_k} = Z_1^- \cup \ldots \cup Z_k^-.
\]
Note that
\[
S_{\vertiii{\cdot}_{\alpha_1,\ldots,\alpha_k}} = S^{+}_{\alpha_1,\ldots,\alpha_k} \cup S^{-}_{\alpha_1,\ldots,\alpha_k}.
\]
If $x \in S^{+}_{\alpha_1,\ldots,\alpha_k}$, then $x_{\ell_j} \in S^{+}_{\alpha_j}$ for some $j \in \{1,\ldots,k\}$, 
and so $\|(T_{\ell_j})^n(x_{\ell_j})\|_{\beta_j} \geq C^{+}_{j,n}$. Thus,
\begin{align*}
\inf_{x \in S^{+}_{\alpha_1,\ldots,\alpha_k}} \vertiii{T^n(x)}_{\beta_1,\ldots,\beta_k}
  &= \inf_{x \in S^{+}_{\alpha_1,\ldots,\alpha_k}} \max_{1 \leq i \leq k} \|(T_{\ell_i})^n(x_{\ell_i})\|_{\beta_i}\\
  &\geq \min_{1 \leq i \leq k} C^{+}_{i,n} \to \infty \ \text{ as } n \to \infty.
\end{align*}
Similarly,
\begin{align*}
\inf_{x \in S^{-}_{\alpha_1,\ldots,\alpha_k}} \vertiii{T^{-n}(x)}_{\beta_1,\ldots,\beta_k}
  &= \inf_{x \in S^{-}_{\alpha_1,\ldots,\alpha_k}} \max_{1 \leq i \leq k} \|(T_{\ell_i})^{-n}(x_{\ell_i})\|_{\beta_i}\\
  &\geq \min_{1 \leq i \leq k} C^{-}_{i,n} \to \infty \ \text{ as } n \to \infty.
\end{align*}
This proves that $T$ is uniformly expansive.
\end{proof}

\begin{corollary}
Let $X$ be a locally convex space and let $T \in GL(X)$.
If $X = X_1 \oplus \cdots \oplus X_m$ is a topological direct sum of $T$-invariant subspaces, then
\[
T \text{ has property } \cP \ \Longleftrightarrow \ T|_{X_i} \text{ has property } \cP \text{ for all } i \in \{1,\ldots,m\}.
\]
\end{corollary}

\begin{proof}
It follows immediately from the previous proposition.
\end{proof}


\section*{Acknowledgement}

The first author is beneficiary of a grant within the framework of the grants for the retraining, modality Mar\'ia Zambrano, 
in the Spanish university system (Spanish Ministry of Universities, financed by the European Union, NextGenerationEU).
The first author was also partially supported by CNPq -- Project {\#}308238/2021-4 and by CAPES -- Finance Code 001.
The three authors were supported by project PID2022-139449NB-I00, funded by MCIN/AEI/10.13039/501100011033/FEDER, UE.
The second and third authors were also supported by Generalitat Valenciana, Project PROMETEU/2021/070.


\end{document}